\documentclass[a4paper,11pt,reqno]{amsart}

\usepackage[utf8]{inputenc}
\usepackage[T1]{fontenc}
\usepackage{lmodern}
\usepackage[english]{babel}
\usepackage{microtype}

\usepackage{amsmath,amssymb,amsfonts,amsthm,esint}
\usepackage{mathtools,accents}
\usepackage{mathrsfs}
\usepackage{aliascnt}
\usepackage{braket}
\usepackage{bm}

\usepackage[a4paper,margin=3cm]{geometry}
\usepackage[citecolor=blue,colorlinks]{hyperref}

\usepackage{enumerate}
\usepackage{xcolor}

\makeatletter
\g@addto@macro\@floatboxreset\centering
\makeatother

\makeatletter
\def\newaliasedtheorem#1[#2]#3{
	\newaliascnt{#1@alt}{#2}
	\newtheorem{#1}[#1@alt]{#3}
	\expandafter\newcommand\csname #1@altname\endcsname{#3}
}
\makeatother

\numberwithin{equation}{section}

\newtheoremstyle{slanted}{\topsep}{\topsep}{\slshape}{}{\bfseries}{.}{.5em}{}

\theoremstyle{plain}
\newtheorem{theorem}{Theorem}[section]
\newaliasedtheorem{proposition}[theorem]{Proposition}
\newaliasedtheorem{lemma}[theorem]{Lemma}
\newaliasedtheorem{corollary}[theorem]{Corollary}
\newaliasedtheorem{counterexample}[theorem]{Counterexample}

\theoremstyle{definition}
\newaliasedtheorem{definition}[theorem]{Definition}
\newaliasedtheorem{question}[theorem]{Question}
\newaliasedtheorem{openquestion}[theorem]{Open Question}
\newaliasedtheorem{conjecture}[theorem]{Conjecture}

\theoremstyle{remark}
\newaliasedtheorem{remark}[theorem]{Remark}
\newaliasedtheorem{example}[theorem]{Example}

\newcommand{\setR}{\mathbb{R}}

\newcommand{\eps}{\varepsilon}

\let\altphi\phi
\let\phi\varphi
\let\varphi\altphi
\let\altphi\undefined

\newcommand{\abs}[1]{\left\lvert#1\right\rvert}

\newcommand{\ft}{\mathrm{t}}

\newcommand{\di}{\mathop{}\!\mathrm{d}}

\newcommand{\loc}{{\rm loc}}

\DeclareMathOperator{\Lip}{Lip}

\DeclareMathOperator{\Cb}{C_b}

\newcommand{\haus}{\mathscr{H}}

\newcommand{\dist}{\mathsf{d}}

\newcommand{\meas}{\mathfrak{m}}

\newcommand{\Test}{{\rm Test}}

\DeclareMathOperator{\CD}{CD}
\DeclareMathOperator{\RCD}{RCD}

\newfont{\tmpf}{cmsy10 scaled 2500}

\DeclareMathOperator*{\essinf}{ess\,inf}

\def\XXint#1#2#3{{\setbox0=\hbox{$#1{#2#3}{\int}$ }
		\vcenter{\hbox{$#2#3$ }}\kern-.6\wd0}}

\begin{document}

\title{On the notion of Laplacian bounds on $\RCD$ spaces and applications}

\author{Nicola Gigli}
\address{SISSA, Via Bonomea 265, 34136 Trieste,  Italy}
\email{ngigli@sissa.it}

\author{Andrea Mondino}
\address{Mathematical Institute, University of Oxford,  Woodstock Rd, Oxford OX2 6GG, UK}
\email{andrea.mondino@maths.ox.ac.uk}

\author{Daniele Semola}
\address{FIM-ETH, R\"amistrasse 101, 8092 Z\"urich, Switzerland}
\email{daniele.semola@math.ethz.ch}

\maketitle

\begin{abstract}
We show that several different interpretations of the inequality $\Delta f\leq\eta$ are equivalent in the setting of $\RCD(K,N)$ spaces. With respect to previously available results in this direction, we improve both  on  the generality of the underlying space and in terms of regularity to be assumed on the function $f$.
Applications are presented.
\end{abstract}

\section{Introduction}

A recurring theme in modern analysis is that of finding appropriate weak interpretations of differential (in)equalities, so that they can be given some useful meaning even if the function being differentiated has no derivatives in the classical sense. We mention here \cite{Calabi58,CheegerGromoll72} and \cite{Wu79} as classical instances of this idea in geometric analysis under lower Ricci curvature bounds, very much related to the subject of the present note.
This is even more important when the underlying geometry is itself non-smooth as in such framework the concept of `classical derivative' typically does not exist at all. In this note we are concerned with the inequality $\Delta f\leq \eta$ in the setting of finite dimensional $\RCD$ spaces and prove, under the sole assumption that $f$ is bounded and lower semicontinuous,  that the following interpretations of such inequality are equivalent:
\begin{itemize}
\item[1)] In the sense of distributions, see \autoref{def:distributions} and \autoref{def:distributions2};
\item[2)] In the sense of comparison with solutions of the associated Poisson equation, see \autoref{def:classicalsupersolution};
\item[3)] In the viscosity sense, see \autoref{def:viscosity};
\item[4)] In the sense of asymptotic for the heat flow, see \autoref{def:heatlaplbounds}.
\end{itemize}
Let us briefly recall the context. The class of $\RCD(K,N)$ spaces consists of metric measure spaces that, in a suitable weak sense, have Ricci curvature $\geq K$ and dimension $\leq N$. It has been introduced in \cite{AmbrosioGigliSavare11-2,Gigli12} as a `Riemannian' counterpart of the original Curvature Dimension condition $\CD(K,N)$ considered in \cite{Lott-Villani09},  \cite{Sturm06I,Sturm06II} (see also \cite{AmbrosioGigliMondinoRajala12, Erbar-Kuwada-Sturm13, AmbrosioMondinoSavare13, CavMil16}). For an overview on the topic and more detailed bibliography we refer to \cite{AmbICM, Villani2017}.

On $\RCD(K,N)$ spaces, the concept of distributional bounds on the Laplacian has been first considered in \cite{Gigli12}, where it was used to state and prove Laplacian comparison estimates for the distance function (see also \cite{CavaMon-APDE2020}). Then in \cite{Gigli-Mondino12} the distributional inequality $\Delta f\leq 0$ on an open set $\Omega$ has been proved to be equivalent to the property of `minimizing the energy among compactly supported non-negative perturbations'. In the recent \cite{MS21}, the notion of Laplacian bound in the viscosity sense has been introduced  in the setting of $\RCD$ spaces and it has been carried out a systematic study of the relations among the various possible such definitions of Laplacian bounds. A main interest in \cite{MS21}, and the motivation for studying viscous bounds, was the study of regularity properties of minimal boundaries in the more restricted class of non-collapsed $\RCD$ spaces (\cite{GDP17} - these are more regular than general $\RCD$ spaces, especially for what concerns properties related to geometric measure theory). It is mostly for this reason that in \cite{MS21} the equivalence among the four concepts above was mainly proven on non-collapsed spaces.

However,  the recent \cite{MS22, G22} are a strong indicator that in fact the equivalence of the above should hold on general $\RCD(K,N)$ spaces: the purpose of this short note is to prove that this is actually the case. While we introduce no radically new idea, the necessary proofs are non-trivial and gather inspiration from a large amount of material recently developed. Also, we pay particular attention in keeping minimal regularity assumptions on the function $f$, thus extending some of the above notions  beyond what was previously known.

As applications, we:
\begin{itemize}
\item[a)] Give an alternative proof of an approximated maximum principle on $\RCD(K,N)$ spaces (see \autoref{thm:apprmaxpr});
\item[b)] Prove sharp Laplacian comparison estimates for the distance from the boundary of a local perimeter minimizer on general $\RCD(K,N)$ spaces. In \cite{MS21} this was obtained on non-collapsed spaces: the generalization is quite straightforward once one has the equivalence of the above (see \autoref{thm:lapdpermin}).
\end{itemize}

We stress that, in the theory of Alexandrov spaces with sectional curvature bounded from below in synthetic sense, a research line aimed at understanding (sub)harmonic functions and their applications was put forward in \cite{Petrunin96,Petrunin03} and further developed in  \cite{ZZ12, ZZ18}. These works have been source of inspiration for the present note and the earlier \cite{MS21,MS22,G22}.
\medskip

We conclude the introduction by mentioning the following question that naturally arose during the development of the present note.

\textbf{Open question:} does the equivalence between Laplacian bounds in the sense of distributions and  in the heat flow sense hold  on general $\RCD(K,\infty)$ metric measure spaces?

\medskip

In order to keep the note short, we assume the reader to be familiar with the basic definitions and concepts from $\RCD$ theory.

\bigskip

\noindent{\bf Acknowledgments} This research has been mostly carried out at the Fields Institute, during the Thematic Program on Nonsmooth Riemannian and Lorentzian Geometry. The authors gratefully acknowledge the warm hospitality and the stimulating atmosphere. 
\\ 
A.\,M.\;received funding from the European Research Council (ERC) under the European Union's Horizon 2020 research and innovation programme, grant agreement No.\;802689 ``CURVATURE'';  for the purpose of Open Access, he has applied a CC BY public copyright licence to any Author Accepted Manuscript (AAM) version arising from this submission.

\section{Some preliminaries and notation}
Throughout the note, $(X,\dist,\meas)$ is an $\RCD(K,N)$ metric measure space, for some $K\in \mathbb{R}$ and $N\in [1,\infty)$. In particular $(X,\dist)$ is a proper metric space and $\meas$ is finite on bounded sets. Given an open subset $\Omega\subset X$, we will denote by $\Cb(\Omega)$ (resp.\;$\Lip_c(\Omega)$) the space of bounded and continuous functions on $\Omega$ (resp.\;the space of Lipschitz functions with compact support contained in $\Omega$).  

\section{Notions of Laplacian bounds and equivalence results}

\begin{definition}[Distributional Laplacian bounds 1]\label{def:distributions}
Let $(X,\dist,\meas)$ be an $\RCD(K,N)$ metric measure space and let $\Omega\subset X$ be an open domain. Let $f\in W^{1,2}_{\loc}(\Omega)$ and $\eta\in\Cb(\Omega)$. We say that $\boldsymbol{\Delta}f\le \eta$ in the sense of distributions if the following holds: for any non-negative function $\phi\in\Lip_c(\Omega)$,
\begin{equation*}
-\int_{\Omega}\nabla f\cdot\nabla \phi\di\meas\le \int_{\Omega}\phi\, \eta\di\meas\, .
\end{equation*}
\end{definition}

More in general, and in analogy with the Euclidean case, it is possible to define Laplacian bounds in the sense of distributions under weaker regularity assumptions on $f$ by ``integrating by parts once more''. Borrowing the notation from \cite{GP19}, for any $\RCD(K,N)$ metric measure space $(X,\dist,\meas)$ we introduce
\begin{equation}
\Test^{\infty}(X):=\{\phi\in D(\Delta)\cap L^{\infty}\, :\, \abs{\nabla \phi}\in L^{\infty}\, ,\Delta\phi\in L^{\infty}\cap W^{1,2} \}
\end{equation}
and, for an open domain $\Omega\subset X$,
\begin{equation}
\Test^{\infty}_{c}(\Omega):=\{\phi\in \Test^{\infty}(X)\, :\, \mathrm{supp}\, \phi\Subset \Omega\}\, .
\end{equation}

\begin{definition}[Distributional Laplacian bounds 2]\label{def:distributions2}
Let $(X,\dist,\meas)$ be an $\RCD(K,N)$ metric measure space and let $\Omega\subset X$ be an open domain. Let $f\in L^1_{\loc}(\Omega)$ and $\eta\in\Cb(\Omega)$. We say that $\boldsymbol{\Delta}f\le \eta$ in the sense of distributions if the following holds: for any non-negative function $\phi\in \Test^{\infty}_{c}(\Omega)$ it holds
\begin{equation}\label{eq:distr3}
\int_{\Omega} f\Delta \phi\di\meas \le \int_{\Omega}\phi\, \eta\di\meas\, .
\end{equation}
\end{definition}

\begin{remark}
If $f\in W^{1,2}_{\loc}(\Omega)$ and $\boldsymbol{\Delta}f\le \eta$ according to \autoref{def:distributions2}, then $\boldsymbol{\Delta}f\le \eta$ also according to \autoref{def:distributions}. The implication can be checked by integrating by parts in \eqref{eq:distr3} and recalling the density of $\Test_c^{\infty}(\Omega)$ into $\Lip_c(\Omega)$.
\end{remark}

\begin{remark}
In the very recent \cite{PengZhangZhu22} a variant of the classical Weyl lemma on Euclidean spaces is obtained for $\RCD(K,N)$ metric measure spaces. Among the corollaries, there is the statement that if $f\in L^{1}_{\loc}(\Omega)$ verifies $\boldsymbol{\Delta}f\le \eta$ and $\boldsymbol{\Delta}f\ge \eta$ according to \autoref{def:distributions2} for some $\eta\in \Cb(\Omega)$ (actually $\eta\in L^{2}_{\loc}(\Omega)$ would be enough), then $f\in W^{1,2}_{\loc}(\Omega)$. See in particular \cite[Corollary 1.5]{PengZhangZhu22}.
\end{remark}

\begin{remark}\label{re:sobreg}
Thanks to the existence of $\Test^{\infty}_{c}(\Omega)$ cut-off functions in  $\RCD(K,N)$ metric measure spaces (these can be constructed by following verbatim the proof of \cite[Lemma 3.1]{Mondino-Naber14} and choosing $f:[0,1] \to [0,1]$ of class $C^3$ in that argument) it is possible to prove that if $f\in L^{\infty}_{\loc}(\Omega)$ satisfies $\boldsymbol{\Delta }f\le \eta$ according to \autoref{def:distributions2}, then $f\in W^{1,2}_{\loc}(\Omega)$. The proof is analogous to the classical one in the Euclidean setting. There are two main steps: an energy estimate and a regularization procedure. The energy estimate builds upon the identity
\begin{equation}\label{eq:idenergy}
\int_{\Omega}\abs{\nabla (\phi f)}^2\di\meas=\int_{\Omega}f^2\abs{\nabla \phi}^2\di\meas+\int_{\Omega}\nabla (\phi^2f)\cdot\nabla f\di\meas\, ,
\end{equation}
valid for any $\phi \in \Lip_c(\Omega)$ and for $f\in W^{1,2}_{\loc}(\Omega)$. In particular, taking a cut-off function $\phi\in \Test^{\infty}_{c}(\Omega)$, the second term at the right hand side in \eqref{eq:idenergy} can be estimated thanks to the distributional Laplacian bound. The regularization is via heat flow, see for instance \cite[Lemma 3.6]{PengZhangZhu22}.
\end{remark}

\begin{proposition}\label{prop:lsc+Lebesgue}
Let $(X,\dist,\meas)$ be an $\RCD(K,N)$ metric measure space, for some $K\in \setR, N\in [1,\infty)$, and let $\Omega\subset X$ be an open domain. Let $f\in W^{1,2}_{\loc}(\Omega)$ and $\eta\in\Cb(\Omega)$. Assume that $\boldsymbol{\Delta}f\le \eta$ in the sense of distributions, then:
\begin{itemize}
\item[i)] $f$ is locally bounded from below and coincides $\meas$-a.e. with its lower semicontinuous essential envelope $f_*$, defined as
\begin{equation}
f_*(x):=\sup_{r>0}\essinf_{B_r(x)}f\, ,\quad \text{for $\meas$-a.e. $x\in \Omega$}\, ;
\end{equation}
\item[ii)] any point $x\in \Omega$ is a Lebesgue point for $f_*$, namely
\begin{equation}\label{eq:Leb}
\lim_{r\to 0}\fint_{B_r(x)\cap \Omega}\abs{f(y)-f_*(x)}\di\meas(y)=0\, ,\quad \text{if}\,\quad   f_*(x)<+\infty\, ,
\end{equation} 
and 
\begin{equation}
\lim_{r\to 0}\fint_{B_r(x)\cap \Omega}f(y)\di\meas(y)=+\infty\, ,\quad \text{if}\, \quad f_*(x)=+\infty\, ;
\end{equation}
\item[iii)] for any $x\in \Omega$ it holds
\begin{equation}\label{eq:Lebheat}
\lim_{t\to 0} P_t\tilde{f}(x)=f_*(x)\, ,
\end{equation}
for any $x\in\Omega$ and for any function $\tilde{f}:X\to\setR$ with polynomial growth and such that $f=\tilde{f}$ in a neighbourhood of $x$.
\end{itemize}
\end{proposition}

\begin{proof}
Items i) and ii) are well known: the techniques in \cite[Chapter 8]{GT01} based on Moser's iteration can be applied with virtually no modification (see also the proof of \cite[Proposition 10]{Lindqvist-Supersolutions}).
\medskip

In order to obtain iii) when $f_*(x)<\infty$ it is sufficient to combine the Lebesgue point condition \eqref{eq:Leb} with \cite[Lemma 2.54]{MS21}. If $f_*(x)=+\infty$, then by lower semicontinuity for any $M\in\setR$ there exists a neighbourhood $U_M\ni x$ such that $f>M$ a.e. on $U_M$. It follows easily that $\lim_{t\to 0}P_tf(x)\ge M$. Taking the limit for $M\to +\infty$ we obtain that \eqref{eq:Lebheat} holds also in the case $f_*(x)=+\infty$.
\end{proof}

The following definition is a variant of \cite[Definition 14.8]{Bjorn-Bjorn11} (in the previous paper \cite{MS21} these were called `classical supersolutions').

\begin{definition}[Supersolution in the comparison sense]\label{def:classicalsupersolution}
Let $(X,\dist,\meas)$ be an $\RCD(K,N)$ metric measure space and let $\Omega\subset X$ be an open domain. Let $f:\Omega\to\setR$ be a lower semicontinuous function and $\eta\in\Cb(\Omega)$. We say that $f$ is a \textit{supersolution  in the comparison sense} of $\Delta f=\eta$ if the following holds: for any open domain $\Omega'\Subset\Omega$ and for any function $g\in C(\overline{\Omega'})$ such that $\Delta g=\eta$ in $\Omega'$ and $g\le f$ on $\partial\Omega'$ it holds $g\le f$ on $\Omega$.
\end{definition}

\begin{remark}
An easy consequence of the standard existence and regularity for solutions to the Dirichlet problem and of the linearity of the Laplacian is the following: given a continuous function $\eta$ and a lower semicontinuous function $u$, it holds that $u$ is a supersolution in the comparison sense of $\Delta f=\eta$ if and only if, denoting by $v_{\eta}$ a local solution of $\Delta v_{\eta}=\eta$ (possibly on a smaller domain $\Omega'\subset\Omega$), $u-v_{\eta}$ is a supersolution in the comparison sense of $\Delta f=0$ (possibly on a smaller domain $\Omega'\subset\Omega$). 
\end{remark}

\begin{remark}
As discussed for instance in \cite[Remark 9.56]{Bjorn-Bjorn11}, it is not true that a supersolution in the comparison sense of $\Delta f= \eta$ belongs to $W^{1,2}_{\loc}$. Counterexamples can be constructed in the Euclidean space endowed with canonical metric measure structure.
\end{remark}

\begin{remark}\label{re:locsolv} On $\RCD(K,N)$ spaces, $N\in [1,\infty)$, it is easy to establish existence of continuous functions $g$ that are local solutions of $\Delta g=\eta$ for $\eta$ continuous and bounded. For instance, any minimizer of the energy $g\mapsto \int_{U}\tfrac12|\nabla g|^2+g\eta\,\di\meas$ in $W^{1,2}(U)$ subject to some given boundary conditions satisfies $\Delta g=\eta$ in $U$ and is, by Proposition \ref{prop:lsc+Lebesgue}, continuous.
\end{remark}

\begin{definition}[Viscous bounds for the Laplacian]\label{def:viscosity}
Let $(X,\dist,\meas)$ be an $\RCD(K,N)$ metric measure space and let $\Omega\subset X$ be an open and bounded domain. Let $f:\Omega\to\setR$ be   lower semicontinuous and $\eta\in\Cb(\Omega)$. We say that $\Delta f\le \eta$ in the \textit{viscous sense} in $\Omega$ if the following holds. For any $\Omega'\Subset\Omega$ and for any test function $\phi:\Omega'\to\setR$ such that 
\begin{itemize}
\item[(i)] $\phi\in D(\Delta, \Omega')$ and $\Delta\phi$ is continuous on $\Omega'$;
\item[(ii)] for some $x\in \Omega'$ it holds 
$\phi(x)=f(x)$ and $\phi(y)\le f(y)$ for any $y\in\Omega'$, $y\neq x$;
\end{itemize}
it holds
\begin{equation*}
\Delta \phi(x)\le \eta(x)\, .
\end{equation*} 
\end{definition}

\begin{remark}\label{rm:shiftingto0}
An easy consequence of the standard existence and regularity for solutions to the Dirichlet problem and of the linearity of the Laplacian is the following statement: given a continuous function $\eta$ and a lower semicontinuous function $u$, it holds that $\Delta u\le \eta$ in the viscous sense if and only if, denoting by $v_{\eta}$ a local solution of $\Delta v_{\eta}=\eta$ (possibly on a smaller domain $\Omega'\subset\Omega$), it holds that $\Delta (u-v_{\eta})\le 0$ in the viscous sense (possibly on a smaller domain $\Omega'\subset\Omega$). 
\end{remark}

\begin{definition}[Supersolution in the heat flow sense]\label{def:heatlaplbounds}
Let $(X,\dist,\meas)$ be an $\RCD(K,N)$ metric measure space and let $\Omega\subset X$ be an open and bounded domain. Let $f:\Omega\to\setR$ be a bounded lower semicontinuous function  and let $\eta\in\Cb(\Omega)$. We say that $\Delta f\le \eta$ on $\Omega$ in the \textit{heat flow sense} if the following holds.  If we denote by $\tilde{f}:X\to\setR$ the global extension of $f$ such that $\tilde{f}(x):=0$ for any $x\in X\setminus\Omega$, then it holds
\begin{equation*}
\limsup_{t\downarrow 0}\frac{P_t\tilde{f}(x)-\tilde{f}(x)}{t}\le \eta(x)\, ,\quad \text{for every $x\in\Omega$}\, .
\end{equation*}
\end{definition}

\begin{remark}
We notice that the choice of the global extension in \autoref{def:heatlaplbounds} does not play any role in the definition, as soon as the extension has polynomial growth. This is a consequence of \cite[Lemma 2.53]{MS21}, applied to the difference of any two global extensions of $f$ with polynomial growth. 
\end{remark}

\section{Proof of the equivalences}

Let us quote two important results connecting supersolutions (in the comparison sense) of the equation $\Delta f=\eta$ with superminimizers. Under our assumptions, they are direct corollaries of \cite[Corollary 9.6, Theorem 9.24, Theorem 14.10]{Bjorn-Bjorn11} (see also \cite{KinnunenMartio02}), where equivalence of supersolutions with superminimizers of the energy is addressed, and  \cite{Gigli-Mondino12}, establishing the equivalence between the property of being superminimizers for the Dirichlet energy and bounds for the Laplacian in the sense of distributions.  As we already remarked, the extension of the results in \cite{Bjorn-Bjorn11} from the case of superharmonic functions $\Delta f\le 0$ to the case of more general upper Laplacian bounds is harmless. This is due to the linearity of the Laplacian, and the local solvability and regularity for solutions of the Poisson equation on $\RCD(K,N)$ spaces (\autoref{re:locsolv}).

\begin{theorem}\label{thm:supersolsupermin}
Let $(X,\dist,\meas)$ be an $\RCD(K,N)$ metric measure space, for some $K\in \setR, N\in [1,\infty)$, and let $\Omega\subset X$ be an open and bounded domain. Let $f \in W^{1,2}_{\loc}(\Omega)$ and let $\eta\in \Cb(\Omega)$. Assume that $\boldsymbol{\Delta}f\le \eta$ in the sense of distributions. Then $f_*$ is a supersolution in the comparison sense of $\Delta f_*\le \eta$.
\end{theorem}  

\begin{theorem}\label{thm:fromclassictodistri}
Let $(X,\dist,\meas)$ be an $\RCD(K,N)$ metric measure space,  for some $K\in \setR, N\in [1,\infty)$, and let $\Omega\subset X$ be an open and bounded domain. Let $f\in W^{1,2}_{\loc}(\Omega)$ and assume that $f_{*}=f$ $\meas$-a.e. and $f_*$ is a supersolution in the comparison sense of $\Delta f_*\le \eta$ for some function $\eta\in C_b(\Omega)$. Then $\boldsymbol{\Delta}f\le \eta$ in the sense of distributions.
\end{theorem}

By \autoref{thm:supersolsupermin} and \autoref{thm:fromclassictodistri} above, the notions in \autoref{def:distributions} and \autoref{def:classicalsupersolution} are  essentially equivalent, at least under the assumption that $f\in W^{1,2}_{\loc}$.

\medskip

The following is a slight extension of \cite[Lemma 3.23]{MS21}. With respect to \cite{MS21} the Lipschitz regularity requirement is dropped, but the proof is basically unchanged.
It is a key tool for the implication from Laplacian bounds in the viscosity sense to Laplacian bounds in the sense of distributions.

\begin{lemma}[Minimum principle for viscosity super solutions]\label{lemma:maxprincvisco}
Let $(X,\dist,\meas)$ be an $\RCD(K,N)$ metric measure space for some $K\in\setR$ and $N\in [1,\infty)$. Let $\Omega\subset X$ be an open and bounded domain such that $\meas(X\setminus\bar{\Omega})>0$.
Let $f:\Omega\to\setR$ be a lower semicontinuous function such that $\Delta f\le 0$ in the viscous sense. Then
\begin{equation*}
\min_{x\in\Omega} f(x)=\min_{x\in\partial \Omega} f(x)\, .
\end{equation*}
\end{lemma}

\begin{proof}
Let us suppose by contradiction that 
\begin{equation*}
\min_{x\in\Omega} f(x)<\min_{x\in\partial \Omega} f(x)\, .
\end{equation*}
Then the minimum in the left hand side is attained at an interior point $x_0\in\Omega$. In particular
\begin{equation}\label{eq:intmin}
\min_{x\in\partial \Omega}f(x)>f(x_0)\, .
\end{equation}
Consider a solution of the Poisson problem $\Delta v=1$ on $X\supset \Omega'\supset \Omega$ such that $v\ge 0$ on $\Omega$ and 
\begin{equation*}
M:=\max_{\partial \Omega} v\ge \min_{\partial \Omega}v=:m>0\, .
\end{equation*} 
This function can be obtained with an additive perturbation from any solution of $\Delta u=1$ on $\Omega'$, by the local Lipschitz regularity \cite[Theorem 1.2]{Jiang12}.
\\We claim that, for $\eps>0$ sufficiently small, also
\begin{equation*}
f_{\eps}(x):=f(x)-\eps v(x)
\end{equation*}
attains a local minimum at an interior point in $\Omega$.\\
Let us suppose by contradiction that this is not the case. Then, for any $\eps>0$, the global minimum  of  $f_{\eps}$ on $\bar{\Omega}$ is attained  on $\partial \Omega$. In particular there exists $x_{\eps}\in \partial \Omega$ such that
\begin{equation*}
f(x_{\eps})-\eps M\le f(x_{\eps})-\eps v(x_{\eps})=f_{\eps}(x_\eps)\le f_{\eps}(x_{0})\le f(x_0)\, .
\end{equation*}
Hence
\begin{equation*}
\min_{x\in\partial\Omega}f(x)-f(x_0)\le f(x_{\eps})-f(x_0)\le M\eps\, ,\quad\text{for any $\eps>0$}\, ,
\end{equation*}
which yields a contradiction with \eqref{eq:intmin} a soon as $\eps$ is sufficiently small.
\medskip

Let now $\eps>0$ be small enough to get that $f_{\eps}=f-\eps v$ has a local minimum $c\in \setR$ at $\bar{x}\in\Omega$. Note that, by assumption, the function $g:=f-c$ satisfies $\Delta g\leq 0$ in the viscous sense. Using $\eps v$ as a test function in the definition of the bound $\Delta g\leq 0$ in viscous sense,  we infer 
\begin{equation*}
\Delta (\eps v)(\bar{x})\le 0\, .
\end{equation*}
This is a contradiction since $\Delta v=1$ on $\Omega$.
\end{proof}

\begin{proof}[Proof of ``Viscosity supersolution implies supersolution in the comparison sense'']
The implication was established in \cite[Theorem 3.24]{MS21} under stronger assumptions on $f$. We repeat the proof here and indicate the minor modifications needed to deal with the more general case.\\
We claim that if $\Delta f\le \eta$ in the viscous sense, then $f$ is a supersolution in the comparison sense to $\Delta f=\eta$, as in  \autoref{def:classicalsupersolution}. This is a consequence of \autoref{lemma:maxprincvisco}. Indeed, let us consider any open subdomain $\Omega'\Subset\Omega$ and any function $g\in C(\overline{\Omega'})$ such that $\Delta g=\eta$ on $\Omega'$ and $g\le f$ on $\partial\Omega'$.\\
Observe that $h:=f-g$ is lower semicontinuous on $\overline{\Omega'}$ and verifies $\Delta h\le 0$ in the viscous sense on $\Omega'$, since $\Delta f\le \eta$ in the viscous sense and $\Delta g=\eta$. Taking into account \autoref{rm:shiftingto0} we infer by \autoref{lemma:maxprincvisco} that
\begin{equation*}
\min_{x\in\Omega'} h(x)=\min_{x\in\partial \overline{\Omega'}} h(x)\ge 0\, .
\end{equation*}
It follows that $f\ge g$ on $\Omega'$, hence $f$ is a supersolution in the comparison sense of $\Delta f=\eta$. 

\end{proof}

\begin{proof}[Proof of ``Distributional implies heat flow sense'']
By a truncation argument and the strong locality properties of the heat flow, see \cite[Lemma 3.8]{G22} and \cite[Lemma 2.53]{MS21}, it is sufficient to establish the implication in the simplified setting where $f\in W^{1,2}(X)$ satisfies $\boldsymbol{\Delta}f\le \eta$ on $X$ for some continuous function $\eta:X\to\setR$ with bounded support. 
\smallskip

Notice that 
\begin{equation}\label{eq:intf}
P_tf(x)-P_sf(x)\le \int_s^tP_r\eta(x)\di r\, ,
\end{equation}
for any $x\in X$ and any $0<s<t$.\\
By \autoref{prop:lsc+Lebesgue}, $f$ admits a lower semicontinuous representative $f_*$ and $P_sf(x)\to f_*(x)$ as $s\to 0$ for any $x\in X$.
By continuity of $\eta$ and \eqref{eq:intf} we get
\begin{equation}
\limsup_{t\to 0}\frac{P_tf(x)-f_*(x)}{t}\le \limsup_{t\to 0}\frac{1}{t}\int_0^tP_s\eta(x)\di s=\eta(x)\, .
\end{equation}

\end{proof}

\begin{proof}[Proof of ``Heat flow sense implies viscosity sense''] Let $x\in\Omega$ and $\phi$ be a test function at $x$ as in the definition of Laplacian bounds in the viscosity sense. By \cite[Lemma 2.56]{MS21}, for any extension $\tilde{\phi}:X\to\setR$ of $\phi:\Omega'\to \setR$ with polynomial growth it holds
\begin{equation}\label{eq:laplaphitilde}
\lim_{t\to 0}\frac{P_t\tilde{\phi}(x)-\phi(x)}{t}=\Delta\phi(x)\, .
\end{equation}
By the very definition of Laplacian bounds in the heat flow sense it holds
\begin{equation}\label{eq:laplaftilde}
\limsup_{t\to 0}\frac{P_t\tilde{f}(x)-f(x)}{t}\le \eta(x)\, ,
\end{equation}
where $\tilde{f}=f$ on $\Omega$ and $\tilde{f}=0$ on $X\setminus \Omega$.\\
As $\phi\le f$ on $\Omega'$, from \eqref{eq:laplaphitilde} and \eqref{eq:laplaftilde} and \cite[Lemma 2.53]{MS21} we immediately deduce that
\begin{equation}
\Delta\phi(x)\le \eta(x)\, .
\end{equation}
\end{proof}

The above implications are clearly sufficient to establish the equivalence between the notions of Laplacian bounds in distributional, heat flow and viscosity sense, under mild regularity assumptions. Below we provide a direct proof of the implication from Laplacian bounds in the heat flow sense to Laplacian bounds in the sense of distributions under the assumption that $f$ is bounded, which is of independent interest.
\smallskip

The key tool in the proof is the stability of upper Laplacian bounds under the Hopf-Lax semigroup, valid on general $\RCD(K,\infty)$ spaces, see \cite[Section 4]{MS21} and \cite[Lemma 4.8]{G22} for the present phrasing.

\begin{lemma}\label{lemma:stabilityhopflax}
Let $(X,\dist,\meas)$ be an $\RCD(K,\infty)$ metric measure space for some $K\in\setR$. Let $f:X\to\setR$ be Borel and bounded. Let $t>0$ and assume that for some $x,y\in X$ it holds
\begin{equation}
\mathcal{Q}_tf(x)=f(y)+\frac{\dist^2(x,y)}{2t}\, .
\end{equation} 
Then 
\begin{equation}
\limsup_{s\to 0}\frac{P_s\mathcal{Q}_tf(x)-\mathcal{Q}_tf(x)}{s}\le \limsup_{s\to 0}\frac{P_sf(y)-f(y)}{s}-K\frac{\dist^2(x,y)}{t}\, .
\end{equation}
\end{lemma}

\begin{proof}[Proof of ``Heat flow sense implies distributional'']
The idea of the proof is the following. After running the Hopf-Lax semigroup for time $t>0$, $\mathcal{Q}_tf$ has the same upper Laplacian bounds of $f$ in the heat flow sense, up to correction terms going to $0$ as $t\downarrow 0$. Moreover, $\mathcal{Q}_tf$ has Laplacian bounded from above in the sense of distributions, by the Laplacian comparison. Hence the (sharper) Laplacian bounds in the heat flow sense  improve to Laplacian bounds in the sense of distributions. The conclusion follows exploiting the stability of Laplacian bounds in the sense of distributions by taking the limit as $t\downarrow 0$.
\medskip

Assume for the moment that $\Omega=X$, i.e.\ that $f:X\to\setR$ is lower semicontinuous and bounded, $\eta:X\to\setR$ is continuous and that $\Delta f\le \eta$ in the heat flow sense. Borrowing the notation from \cite[Section 4]{G22}, for any $x\in X$ and $t>0$ we let $F_t(x)$ be a minimizer in the variational definition of $\mathcal{Q}_tf(x)$, namely, 
\begin{equation}
\mathcal{Q}_tf(x)=f(F_t(x))+\frac{\dist^2(x,F_t(x))}{2t}\, .
\end{equation}
Notice that such a minimizer exists for every $x\in X$ by local compactness, boundedness and lower semicontinuity of $f$. Moreover, $F_t(x)$ is uniquely defined for $\meas$-a.e. $x\in X$, by \cite[Theorem 4.9]{G22}. 
\smallskip

By \autoref{lemma:stabilityhopflax}, for every $x\in X$ it holds
\begin{equation}
\limsup_{s\to 0}\frac{P_s\mathcal{Q}_tf(x)-\mathcal{Q}_tf(x)}{s}\le \limsup_{s\to 0}\frac{P_sf(F_t(x))-f(F_t(x))}{s}-K\frac{\dist^2(x,F_t(x))}{t}\, .
\end{equation}
Since $f$ has Laplacian bounded from above by $\eta$ in the heat flow sense, we deduce that
\begin{equation}
\limsup_{s\to 0}\frac{P_s\mathcal{Q}_tf(x)-\mathcal{Q}_tf(x)}{s}\le \eta(F_t(x))-K\frac{\dist^2(x,F_t(x))}{t}\, ,
\end{equation}
for every $x\in X$. \\
By the Laplacian comparison for $\RCD(K,N)$ spaces, $\mathcal{Q}_tf$ has Laplacian locally bounded from above in the sense of distributions. Hence by \cite[Lemma 3.2]{G22} it holds
\begin{equation}\label{eq:boundt}
\boldsymbol{\Delta}\mathcal{Q}_tf\le \eta_t\meas\, ,
\end{equation}
where we set
\begin{equation}
\eta_t(x):=  \eta(F_t(x))-K\frac{\dist^2(x,F_t(x))}{t}\, ,
\end{equation}
for any $x\in X$ and any $t>0$.\\
By standard arguments (see for instance \cite[Section 3]{AmbrosioGigliSavare11}), the functions $\abs{\eta_t}$ are locally uniformly bounded, with bounds independent of $t\in (0,1)$. \\
Analogously, the functions $\mathcal{Q}_tf$ are locally uniformly bounded with bounds independent of $t\in (0,1)$ and $\mathcal{Q}_tf\to f$ as $t\downarrow 0$ in the $\meas$-a.e.\;sense. Moreover, the functions $\mathcal{Q}_tf$ have locally bounded $W^{1,2}$-norms, uniformly w.r.t. $t$ (by \autoref{re:sobreg}). Hence, by lower semicontinuity of the $W^{1,2}$ energy, $f\in W^{1,2}_{\loc}$.
By the discussion around \cite[Equation (3.30)]{G22} we can pass to the limit the bounds \eqref{eq:boundt} for a sequence $t_n\downarrow 0$ to obtain  that $f$ has Laplacian locally bounded from above in the sense of distributions (with a possibly non sharp upper bound).\\
To conclude, the non sharp upper Laplacian bounds in the sense of distributions can be combined, again by by \cite[Lemma 3.2]{G22}, with the upper Laplacian bounds in the heat flow sense to obtain that $\Delta f\le \eta$ in the sense of distributions.

To deal with the case $\Omega$  bounded open subset of $X$ we argue as in \cite[Lemma 3.7]{G22}.\\ 
Let $s:=\sup_\Omega f$, $\Omega'\Subset\Omega$ and for $C\gg0$ consider the function  $f':=f+C\dist^2(\cdot,\Omega')$. By the (distributional) Laplacian comparison properties of the distance and the implication `distributional sense $\Rightarrow$ heat flow sense' that we already established, 
we know that $\Delta f'\leq\eta'$ in the heat flow sense for some $\eta'\in \Cb(\Omega)$ with $\eta'=\eta$ in $\Omega'$. This last identity follows from the locality properties of the heat flow, see \cite[Lemma 3.8]{MS21}.\\ 
It is clear that for $C$ sufficiently large we have $f>s$ in a neighbourhood of $\partial\Omega$ and then that the function $f'':=f'\wedge s$, defined as $s$ outside $\Omega$ satisfies $\lim_{t\downarrow0}\frac{P_tf''(x)-f''(x)}{t}=0$ in the interior of $\{f''=s\}$ and 
\begin{equation}
\limsup_{t\downarrow0}\frac{P_tf''(x)-f''(x)}{t}\leq \limsup_{t\downarrow0}\frac{P_tf'(x)-f''(x)}{t}\, ,\quad  \text{on}\quad  \{f'\leq s\}\subset\Omega\, . 
\end{equation}
Hence $\Delta f''\leq \eta''$ in the heat flow sense for some $\eta''\in \Cb(X)$ equal to $\eta$ on $\Omega'$. By what  proved above we conclude that $\boldsymbol{\Delta} f''\leq \eta''$ in 
the sense of distributions on the whole $X$. Hence  $\boldsymbol{\Delta}f\leq \eta$ on $\Omega'$, and then also on the whole $\Omega$, by the locality properties of the distributional Laplacian and the arbitrariness of $\Omega'\Subset\Omega$ (see also \cite[Proposition 4.17]{Gigli12}).
\end{proof}

\section{Applications}

With the tools developed in this note and in the previous \cite{G22,MS21}, we can recover a version of the classical approximate maximum principle in the viscosity theory of PDEs. We refer to \cite[Theorem 4.2]{ZhangZhu16} for a different approach based on Kato's inequality.

\begin{theorem}\label{thm:apprmaxpr}
Let $(X,\dist,\meas)$ be an $\RCD(K,N)$ metric measure space, for some $K\in \setR$ and $N\in [1,\infty)$. Let $\Omega\subset X$ be a bounded domain and $f\in W^{1,2}(\Omega)\cap \Cb(\Omega)$. Assume that $\boldsymbol{\Delta }f\le C$ in the sense of distributions on $\Omega$ for some $C\in\setR$. If $x\in\Omega$ is a minimum point of $f$, then there exists a sequence $\Omega\ni x_n\to x$ such that 
\begin{itemize}
\item[(i)] $x_n$ are Lebesgue points of the absolutely continuous part $\Delta^{\mathrm{ac}}f$ of $\boldsymbol{\Delta}f$;
\item[(ii)] $[\Delta^{\mathrm{ac}}f(x_n)]_{-}\to 0$ as $n\to \infty$;
\item[(iii)] $x_n$ are Lebesgue points for $\abs{\nabla f}$ and $\abs{\nabla f}(x_n)\to 0$ as $n\to\infty$.
\end{itemize}

\end{theorem}

\begin{proof}
Up to changing $f$ with $f+\eps\dist_x^2$ for $\eps$ arbitrary small we can assume that $f$ has a strict minimum point at $x$. The additional error terms could be handled with arguments analogous to those that we are going to employ below.\\
By the Lebesgue differentiation theorem, $\meas$-a.e.\;point is a Lebesgue point of $\Delta^{\mathrm{ac}}f$. Moreover, by \cite[Theorem 5.4]{Ambrosioetalembedding} (see also \cite{Cheeger99}), $\meas$-a.e. point $x\in\Omega$ is a harmonic point for $f$ according to \cite[Definition 5.2]{Ambrosioetalembedding}. In particular, $x$ is a Lebesgue point of $\abs{\nabla f}$ and any blow-up of the function $f$ at $x$ is a splitting function on a tangent space of $(X,\dist,\meas)$ with slope $\abs{\nabla f}(x)$, unless $\abs{\nabla f}(x)=0$, in which case the blow up is the constant function $0$.\\ 
By \cite[Corollary 4.5]{MS22} for any $n\ge 1$ there exist $a_n<1/n$ and $y_n\in\Omega$ with $\dist(x,y_n)<1/n$ such that the function
\begin{equation}\label{eq:deffn}
f_n(z):=f(z)+a_n\dist^2(z,y_n)
\end{equation}
achieves a local minimum at a point $x_n$ which is a Lebesgue point of $\Delta^{\mathrm{ac}}f$ and a harmonic point for $f$ such that $\dist(x,x_n)<1/n$. We claim that the sequence $(x_n)$ verifies the conclusion in the statement.\\ 
Indeed by local minimality and \cite[Lemma 2.56]{MS21}
\begin{equation}\label{eq:1sign}
\liminf_{t\to 0}\frac{P_tf_n(x_n)-f_n(x_n)}{t}\ge 0\, .
\end{equation}
On the other hand, the Laplacian comparison for $\RCD(K,N)$ spaces and \cite[Lemma 2.54]{MS21} yield
\begin{align}\label{eq:2sign}
\limsup_{t\to 0}\frac{P_tf_n(x_n)-f_n(x_n)}{t}\le &\,  \Delta^{\mathrm{ac}}f(x_n)+a_n\limsup_{t\to 0}\frac{P_t\dist^2(\cdot,y_n)(x_n)-\dist^2(x_n,y_n)}{t}\\
\le &\,  \Delta^{\mathrm{ac}}f(x_n) +C_{K,N}a_n\, ,
\end{align}
where $C_{K,N}$ is a constant depending only on the lower Ricci and upper dimension bounds coming from the Laplacian comparison in \cite{Gigli12}.\\
Combining \eqref{eq:1sign} and \eqref{eq:2sign} we obtain 
\begin{equation}
 \Delta^{\mathrm{ac}}f(x_n)\ge -C_{K,N}a_n\, ,
\end{equation} 
for any $n\ge 1$, and (ii) follows.\\
Moreover, by local minimality again, any blow-up of the function $f_n$ introduced in \eqref{eq:deffn} at $x_n$ has a minimum in the base point of the tangent space. However, any blow-up of $f_n$ at $x_n$ is the sum of a splitting function with slope $\abs{\nabla f}(x_n)$, because $x_n$ is a harmonic point of $f$, and a Lipschitz function with Lipschitz constant bounded above by $2a_n\dist(x_n,y_n)\le 4/n^2$. It is easy to conclude that $\abs{\nabla f_n}(x)\le 4/n^2$ and (iii) immediately follows.
\end{proof}

In \cite{MS21}, the second and third author obtained a sharp Laplacian comparison for the distance function from a locally perimeter minimizing set on $\RCD(K,N)$ metric measure spaces $(X,\dist,\haus^N)$. The restriction on the reference measure being $\haus^N$ can be dropped thanks to the equivalence between different notions of Laplacian bounds obtained in the present note.
\medskip

Let us define the comparison function $\ft_{K,N}: I_{K,N}\to \setR$ as
\begin{equation}\label{eq:deftKN}
\begin{split}
\ft_{K,N}(x)&:=
\begin{cases}
- \sqrt{ K(N-1)} \tan\big(\sqrt{\frac{K} {N-1}} x \big)\,& \quad \text{if } K>0\\
\quad 0\, & \quad \text{if } K=0 \\
\sqrt{-K(N-1)} \tanh\big(\sqrt{\frac{-K} {N-1}} x \big)\,& \quad \text{if } K<0 \;, 
\end{cases} 
\\
I_{K,N}&:=
\begin{cases}
\big( - \frac{\pi}{2} \sqrt{\frac{N-1}{K}},  \frac{\pi}{2} \sqrt{\frac{N-1}{K}}  \big)\,& \quad \text{if } K>0\\
\quad \setR \,& \quad \text{if } K\leq 0 \; .
\end{cases} 
\end{split}
\end{equation}

\begin{theorem}\label{thm:lapdpermin}
Let $(X,\dist,\meas)$ be an $\RCD(K,N)$ metric measure space, for some $K\in \setR$ and $N\in [1,\infty)$. Let $E\subset X$ be a set of locally finite perimeter and assume that it is a local perimeter minimizer. Let $\dist_{\overline{E}}:X\setminus \overline{E}\to[0,\infty)$ be the distance function from $\overline{E}$. Then
\begin{equation}\label{eq:meanglob}
\Delta \dist_{\overline{E}}\le \ft_{K,N} \circ \dist_{\overline{E}}\, \quad\text{on $X\setminus\overline{E}$}\, ,
\end{equation}
where $ \ft_{K,N}$ is defined in \eqref{eq:deftKN}.
If $\Omega\subset X$ is an open domain and $E\subset X$ is locally perimeter minimizing in $\Omega$, then setting
\begin{equation}\label{eq:defK}
{\mathcal K}:=\{x\in X\, :\, \exists\,  y\in \Omega\cap \partial E\, \text{ such that }\, \dist_{\overline{E}}(x)=\dist(x,y)\}\, ,
\end{equation}
it holds
\begin{equation}\label{eq:meangen}
\Delta \dist_{\overline{E}}\le \ft_{K,N}\circ \dist_{\overline{E}}\,\, \quad\text{on any open subset $\Omega'\Subset \left(X\setminus\overline{E}\right)\cap {\mathcal K}$}\, .
\end{equation}

\end{theorem}

\begin{proof}
The proof requires only a very minor adjustment with respect to the one of \cite[Theorem 5.2]{MS21}. Namely, borrowing the notation introduced therein, in Step 3 of the proof in \cite{MS21} we set
\begin{equation}
\hat{\psi}:=\psi-\delta \dist_x^2\, .
\end{equation}
All the properties from (i') to (iv') of $\hat{\psi}$ remain valid with the following two exceptions: 
\begin{itemize}
\item[a)] the function $\hat\psi$ does not belong to the domain of the Laplacian. However, it has measure valued Laplacian locally bounded from below by the Laplacian comparison for $\RCD(K,N)$ spaces;
\item[b)] only the lower Laplacian bound $\boldsymbol{\Delta}\hat{\psi}>\eps'$ holds in (iv') and it has to be intended in the sense of distributions.
\end{itemize}
Analogous comments hold for the global extension $\overline{\psi}$ of $\hat{\psi}$. We remark that the upper Laplacian bound on $\overline{\psi}$ played no role in the subsequent arguments in \cite{MS21} that can be carried over without further modifications in the present setting.
\end{proof}

\bibliographystyle{siam}
\bibliography{biblio}

\def\cprime{$'$} \def\cprime{$'$}

\end{document}